\documentclass{article}

\usepackage[preprint]{neurips_2020}
\usepackage{amsmath,amsthm,amsfonts,amssymb,amscd}
\newtheorem{theorem}{Theorem}[section]
\newtheorem{definition}[theorem]{Definition}
\newtheorem{corollary}[theorem]{Corollary}
\newtheorem{lemma}[theorem]{Lemma}
\newtheorem{assumption}[theorem]{Assumption}
\newtheorem{remark}[theorem]{Remark}

\usepackage[utf8]{inputenc} % allow utf-8 input
\usepackage[T1]{fontenc}    % use 8-bit T1 fonts
\usepackage{hyperref}       % hyperlinks
\usepackage{url}            % simple URL typesetting
\usepackage{booktabs}       % professional-quality tables
\usepackage{amsfonts}       % blackboard math symbols
\usepackage{nicefrac}       % compact symbols for 1/2, etc.
\usepackage{microtype}      % microtypography

\usepackage{xcolor}
\title{Global Convergence of Deep Galerkin and PINNs Methods for Solving Partial Differential Equations}

\author{%
  Deqing Jiang\\
  Mathematical Institute\\
  University of Oxford\\
  Oxford, UK, OX2 6GG \\
  \texttt{jiangd@maths.ox.ac.uk} \\
   \And
Justin Sirignano \\
Mathematical Institute\\
  University of Oxford\\
  Oxford, UK, OX2 6GG \\
\texttt{sirignano@maths.ox.ac.uk} \\
\And
Samuel N. Cohen \\
Mathematical Institute\\
  University of Oxford\\
  Oxford, UK, OX2 6GG \\
\texttt{cohens@maths.ox.ac.uk} \\
}

\begin{document}

\maketitle

\begin{abstract}
  Numerically solving high-dimensional partial differential equations (PDEs) is a major challenge. Conventional methods, such as finite difference methods, are unable to solve high-dimensional PDEs due to the curse-of-dimensionality. A variety of deep learning methods have been recently developed to try and solve high-dimensional PDEs by approximating the solution using a neural network. In this paper, we prove global convergence for one of the commonly-used deep learning algorithms for solving PDEs, the Deep Galerkin Method (DGM). DGM trains a neural network approximator to solve the PDE using stochastic gradient descent. We prove that, as the number of hidden units in the single-layer network goes to infinity (i.e., in the ``wide network limit"), the trained neural network converges to the solution of an infinite-dimensional linear ordinary differential equation (ODE). The PDE residual of the limiting approximator converges to zero as the training time $\rightarrow \infty$. Under mild assumptions, this convergence also implies that the neural network approximator converges to the solution of the PDE. A closely related class of deep learning methods for PDEs is Physics Informed Neural Networks (PINNs). Using the same mathematical techniques, we can prove a similar global convergence result for the PINN neural network approximators. Both proofs require analyzing a kernel function in the limit ODE governing the evolution of the limit neural network approximator. A key technical challenge is that the kernel function, which is a composition of the PDE operator and the neural tangent kernel (NTK) operator, lacks a spectral gap, therefore requiring a careful analysis of its properties. 
\end{abstract}

\section{Introduction}

Deep learning methods have become widely-used for solving high-dimensional PDEs and modeling physics data governed by PDEs. Although low-dimensional PDEs can be efficiently solved with existing numerical techniques, such as finite difference methods, high-dimensional PDEs are
computationally intractable due to the curse-of-dimensionality. An alternative approach that has been widely employed is to approximate the PDE solution with a neural network and then train the neural network with stochastic gradient descent to satisfy the PDE and its boundary conditions
-- e.g., the deep Galerkin method (DGM) in \citet{sirignano2018dgm}. A similar method -- physics-informed neural networks (PINNs) in \cite{raissi2019physics} -- was developed to model physics data by training a neural network to both satisfy the corresponding governing PDE and match a sparse set of experimental observations. Both methods share a common feature of training a neural network with (stochastic) gradient descent to satisfy an objective function with PDE. Numerous other articles in the literature have also explored solving PDEs with neural networks (see \cite{beck2020overview} for an overview). Solving PDEs with neural network approximators is a natural idea that has been considered in different forms for decades, for instance, \cite{lee1990neural}, \cite{lagaris1998artificial}, \cite{malek2006numerical} and \cite{rudd2013solving}. These papers propose to use neural networks to solve differential equations by estimating neural-network solutions on an a priori fixed mesh. However, in this paper, the algorithm we mainly discuss (DGM) is mesh-free and hence it can be applied to solve high-dimensional PDE problems.

Although it is clear -- due to the universal approximation properties of neural networks \citep{hornik1991approximation} -- that \emph{there exists} a neural network which can approximate the solution to a given PDE (interpreting the PDE solution as a function in the Sobolev space $\mathcal{H}^2$), the convergence of the neural network \emph{when trained with gradient descent} to the PDE solution has not been previously proven. 

Convergence analysis for the optimization of neural network approximators to PDEs must address several mathematical challenges. First, the neural network is non-convex in its parameters, which is further exacerbated by applying a PDE operator to the neural network in the objective function. Consequently, as the number of hidden units $\rightarrow \infty$, the standard neural tangent kernel (NTK) does not arise \citep{jacot2018neural}. Instead, the kernel function involves the PDE operator, requiring the development of new mathematical analysis. Finally, as is also true for the standard NTK setting, the kernel lacks a spectral gap, which makes analysis of infinite-dimensional systems (such as approximators to PDEs) challenging. Our proof leverages a careful analysis of the eigendecomposition of the limit ODE and its kernel function. 

For both the DGM and PINN algorithms, we prove that, as the number of hidden units in the single-layer network goes to infinity (i.e., in the ``wide network limit"), the trained neural network converges to the solution of an infinite-dimensional linear ODE. The PDE residual of the limiting approximator converges to zero as the training time $\rightarrow \infty$. Under mild assumptions, this convergence also implies that the neural network approximator converges to the solution of the PDE. \cite{wang2022and} prove that, under some assumptions, as the number of neurons goes to infinity, the training process of PINNs  will converge to a process characterized by a kernel matrix. However, \cite{wang2022and} does not prove global convergence of the neural network approximator.

Our paper provides a rigorous mathematical analysis of the DGM and PINN training process for solving PDEs with neural networks. In summary, the key contributions of this paper are:
\begin{enumerate}
    \item We prove that as the number of hidden units in the neural network $\rightarrow \infty$ (i.e., in the ``wide network limit), the training process of the neural approximator trained to minimize the PDE residual converges to an infinite-dimensional linear ODE characterized by a kernel function. 
    \item The kernel function is different than the standard NTK kernel and involves the PDE operator.
    \item We prove that even though the kernel is only positive semi-definite and there is no spectral gap, the objective function (i.e., the PDE residual of the wide-limit neural network) converges to zero as the training time $t \rightarrow \infty$. This result establishes global convergence. Furthermore, under an additional mild assumption, the wide-limit neural network converges to the PDE solution. 
\end{enumerate}

The paper is organized as follows. In Section \ref{sec2}, we introduce the class of PDEs that will be considered. Section \ref{sec3} describes the neural network training algorithm for solving PDEs and then proves that the neural network approximator converges to the limit ODE as the number of hidden units $\rightarrow \infty$. Section \ref{sec5} studies the properties of the kernel function that characterizes the limit ODE. Then, we prove that the PDE residual converges to zero as the training time $\rightarrow \infty$. Then, it is proven that -- with an additional mild assumption on the PDE -- the wide-limit neural network also converges to the PDE solution. In Section \ref{sec7}, we prove global convergence for the PINN algorithm. Lemmas, corollaries, and theorems are presented in the main part of the paper. All mathematical proofs are in the Appendix \ref{App}.

\section{Mathematical Framework} \label{sec2}
We will study the convergence of neural network algorithms -- such as DGM and PINNs -- for solving PDEs. In particular, we will analyze the convergence of such algorithms for the following class of second-order linear PDEs with Dirichlet boundary conditions: \begin{align}
    \begin{cases} \label{pde}
        \mathcal{A}v&= h,\quad \text{in}\,\, \Omega\\
        \,\,\,\,v&= f,\quad \text{on} \,\, \partial \Omega,
    \end{cases}
\end{align}
where $\Omega\in \mathbb{R}^d$ is a compact set with a smooth boundary.  We will study strong Sobolev solutions to the PDE \eqref{pde1}; that is, we are interested in solutions $u\in \mathcal{H}^{2}$, where for a finite measure $\mu$ on $\Omega$,
\begin{equation}\label{eq:sobolev}
    \mathcal{H}^p = \bigg\{f\in L^2(\Omega, \mu): \|f\|_{\mathcal{H}^p}:=\Big( \sum_{|\alpha|<p} \|D_\alpha f\|_{L^2} \Big)<\infty\bigg\},
\end{equation}
where $Du$ is the weak derivative of $u$ (see \cite{evans2022partial}).  We assume $\mu$ is equivalent to Lebesgue measure, and the logarithm of its Radon--Nikodym derivative is bounded (which ensures it generates the same $\mathcal{H}^2$ space as Lebesgue measure). For notational convenience, we will write $\mathcal{H}^2_{(0)} = \mathcal{H}^2\cap \mathcal{H}^1_0$, representing the $\mathcal{H}^2$ functions with zero value on the boundary, equipped with the $\mathcal{H}^2$ norm.

We make the following (standard) assumptions on our problem:
\begin{assumption}[Smoothness of the boundary $\Omega$] \label{a1}
The boundary $\partial \Omega$ is $C^{3, \alpha}$ for some $\alpha\in(0,1)$; i.e., three times continuously differentiable with $\alpha$-H\"older continuous derivatives of order 3.
\end{assumption}
\begin{assumption}[Auxiliary function $\eta$]\label{auxiliaryfunction} \label{a2}
There exists a (known) function $\eta \in C^3_b(\mathbb{R}^n)$, which satisfies $\eta>0 $ in $\Omega$, and $\eta=0$ on $\partial \Omega$. Furthermore, its first order derivative does not vanish at the boundary (that is, for  $x\in \partial \Omega$ and $\mathbf{n}_{x}$ an outward unit normal vector at $x$, we have  $ \nabla \eta(x)\cdot  \mathbf{n}_{x} \neq 0$).
\end{assumption}
\begin{assumption}[Interpolation of the boundary condition function]
    There exists a (known) function $\bar f \in \mathcal{H}^{2}$ such that $\bar f|_{\partial\Omega}=f$. In the rest of this paper, we identify $f$ with its extension $\bar f$ defined on $\overline{\Omega}$ for notational simplicity.
\end{assumption}
%Given that $f$ is smooth enough to interpolate on $\Bar{\Omega}$, 
We can reformulate the PDE as \begin{align} \label{pde1}
    \begin{cases}
        \mathcal{A}u&= {g},\quad \text{in}\,\, \Omega\\
        \,\,\,\,u&= 0,\quad \text{on} \,\, \partial \Omega,
    \end{cases}
\end{align}
where $u:=v-f$ and ${g}:=h-\mathcal{A}f$. Finally, we assume that the PDE operator satisfies a certain type of Lipschitz condition:
\begin{assumption}[Lipschitz condition]
    \label{assume_1}
There exists a constant $k>0$ such that for any $f_1, f_2 \in \mathcal{H}^2$ and any $x \in \Omega$, the linear operator $\mathcal{A}$ satisfies
\begin{align} \label{Lip}
    \begin{split}
    |\mathcal{A}f_1(x)-\mathcal{A}f_2(x)| &\leq k \bigg [ \sum_{0 \leq |\alpha|\leq 2} |D_\alpha f_1(x)-D_\alpha f_2(x)| \bigg ].
\end{split}
\end{align}
\end{assumption}
\section{Deep Learning Algorithms for Solving Partial Differential Equations} \label{sec3}

Deep learning PDE algorithms -- such as DGM and PINNs -- train a neural network approximator to satisfy the PDE and its boundary conditions using either gradient descent or stochastic gradient descent. 

Consider the following single-layer neural network with $N$ hidden units $S^N$: \begin{align}
    S^N(x; \theta^N) =\frac{1}{N^\beta}\sum_{i=1}^N c^i \sigma(w^ix+b^i),
\end{align}
where $\frac{1}{N^\beta}$ is a normalization factor and $\frac{1}{2}<\beta<1$.  We train a neural network $Q^N$ to approximate the solution $u$ to the PDE where
\begin{align}
    Q^N(x; \theta^N):=\eta(x)S^N(x; \theta^N)=\eta(x)\cdot \frac{1}{N^{\beta}}\sum_{i=1}^N c^i \sigma(w^i x+b^i).
\end{align}
 $\eta(x)$ is a fixed function which vanishes on the boundary $x \in \partial \Omega$; therefore $Q^N$ automatically satisfies the boundary conditions of the PDE (\ref{pde1}). This method, which was first introduced by \cite{mcfall2009artificial}, simplifies the training of the neural network model. The parameters $\theta^N = (c^i, w^i, b^i)_{i=1}^N$ must be trained using gradient descent to satisfy the PDE in the interior of the domain. Specifically, we will minimize the PDE residual error for the neural network by minimizing the following objective function:

\begin{align} \label{J}
    J(\theta^N)=\|\mathcal{A}Q^N-g\|_{L^2(\mu)}^2:=\int_\Omega [\mathcal{A}Q^N(x; \theta^N)-g(x)]^2 d\mu(x),
\end{align}
where $\mu$ is a sampling measure (satisfying the regularity assumptions stated after \eqref{eq:sobolev} above). If the residual term $\mathcal{A}Q^N(x; \theta^N)-g(x)$ equals zero for all $x \in \Omega$, then $Q^N=u$ is the solution of the PDE. We will minimize the objective function (\ref{J}) using gradient descent with clipping. The gradient of (\ref{J}) is:
\begin{align}
    \nabla_\theta J(\theta^N) = \int_\Omega [\mathcal{A}Q^N(x; \theta^N)-g(x)] \nabla_\theta \mathcal{A}Q^N(x; \theta^N)d\mu(x).
\end{align}
Gradient clipping is widely used in deep learning, see for instance \cite{zhang2019gradient}, \cite{pascanu2013difficulty} and chapters 10 and 11 of \cite{goodfellow2016deep}. The continuous-time gradient descent training with clipping is given by:
\begin{align} \label{9}
    \frac{d\theta^N_t}{dt}=-\alpha^N G^N(\theta_t^N),
\end{align}
where the learning rate is $\alpha^N=N^{2\beta-1}$ and \begin{align}
    G^N(\theta_t^N)=-\int_\Omega \psi^N(\mathcal{A}Q^N(x; \theta^N_t)-g(x))\Phi^N(\nabla_\theta \mathcal{A}Q^N(x; \theta^N_t))d\mu(x).
\end{align}
Here $\Phi^N$ is a vector function that applies elementwise clipping with the function $\phi^N$. For each entry of vector $\nabla_\theta \mathcal{A}Q^N$, we clip its value with the scalar function $\phi^N$.

In practice, (\ref{9}) can be approximated by discretizing in time and, at each time step, generating Monte Carlo samples from the measure $\mu$ to approximate the integral, which is highly computationally efficient even for high-dimensional PDEs. This is also equivalent to the stochastic gradient descent version of (\ref{9}). Although not investigated in this paper, standard methods can be used to prove that stochastic gradient -- using the correct learning rate -- will converge to the continuous-time gradient flow (\ref{9}); for example, weak convergence analysis such as in \cite{SirignanoSpiliopoulos1} and \cite{SirignanoSpiliopoulos2} could be used. 

Let $Q_t^N(x) = Q^N(x; \theta_t^N)$ be the neural network at training time $t$ with $N$ hidden units. We will analyze the trained neural network $Q_t^N$  as the number of hidden units $N \rightarrow \infty$ and the training time $t \rightarrow \infty$. First, we prove that the trained neural network $Q_t^N$ will converge to the solution of an infinite-dimensional ODE as $N \rightarrow \infty$. That is, in the ``wide limit" where the number of hidden units $\rightarrow \infty$, $Q_t^N$ converges to the solution of an ODE. Then, we prove that the wide-limit neural network converges to the global minimizer of the objective function (with zero PDE residual) as the training time $t \rightarrow \infty$. Under additional mild assumptions, this global minimizer is also a solution to the PDE. These convergence results can also be proven for the PINNs algorithm for solving PDEs; see Section \ref{sec7}. 

Our convergence results will be proven under the following assumptions on the neural network architecture:

\begin{assumption}[Activation function]\label{assume_activation}
The activation function $\sigma\in C_b^4(\mathbb{R})$ is non-constant.
\end{assumption}
\begin{assumption}[Neural network initialization]\label{initialization}
The initialization of the parameters $\theta_0^N$, for all $i \in \{1,2,...,N\}$, satisfies: \begin{itemize}
    \item The parameters $c_0^i$, $w_0^i$, $b_0^i$ are i.i.d. random variables.
    \item The random variables $c_0^i$ are bounded, $|c^i_0|<K_0$, and $\mathbb{E}[c^i_0]=0$.
    \item \label{cybenko}
The distribution of the random variables $w_0^i, b_0^i$ has full support. That is, for any open set $D \subseteq \mathbb{R}^{n+1}$, we have $\mathbb{P}((w_0^i,b_0^i)\in D)>0$.
\item The moments  $\mathbb{E}[|(w_0^i)_k|^3]$ and $\mathbb{E}[|b_0^i|]$ are bounded where $(w_0^i)_k$ is the $k$-element of $w_0^i$.
\end{itemize}
\end{assumption}
\begin{definition} [Smooth clipping function]
\label{psi}
A function class $\{h^N\}_{N\in \mathbb{N}^+}$ forms a family of smooth clipping functions with parameter $\gamma > 0$ if for any $N\in \mathbb{N}^+$  \begin{itemize}
    \item $h^N \in \mathbb{C}^2_b(\mathbb{R})$ is increasing on $\mathbb{R}$.
    \item $|h^N|$ is bounded by $2N^\gamma$.
    \item $h^N(x)=x$ for $x \in [-N^\gamma,N^\gamma]$. \item $|(h^N)'| \leq 1$ on $\mathbb{R}$.
\end{itemize}
\end{definition}
\begin{assumption}
\label{phi}
Functions $\{\psi^N\}_{N\in \mathbb{N}^+}$ and $\{\phi^N\}_{N\in \mathbb{N}^+}$ are families of smooth clipping functions with parameter $\delta$ and $\epsilon-\beta$ where $\epsilon>\delta>0$, $\beta \in (\frac{1}{2},1)$ and $\epsilon+\delta < \frac{1-\beta}{2}.$
\end{assumption}

\begin{lemma} \label{l35}
    There exists a constant $k$ independent of $N$ such that the change of each component of $c^i,w^i,b^i$ from its initial condition (e.g. $|c^i_t-c^i_0|$) is bounded by $kt N^{2\beta-1+\delta+\epsilon-\beta}=ktN^{\beta+\delta+\epsilon-1}$.
\end{lemma}

\subsection{Convergence of the trained neural network as the number of hidden units $N \rightarrow \infty$}
In this section, we analyze the evolution of the neural network $Q^N_t$ as it is trained to minimize the PDE residual in the objective function $J(\theta^N)$. We can prove that, as the number of hidden units $N \rightarrow \infty$, the neural network $Q^N_t(y)$ will converges to the solution $Q_t(y)$ of an infinite-dimensonal linear ODE. By the chain rule, the dynamics of $Q^N$ satisfy
\begin{align} \label{11}
\begin{split}
    \frac{d Q^N_t}{dt}(y)&=\nabla_\theta Q_t^N(y) \cdot \frac{d\theta^N}{dt}\\
    &=-\int_\Omega \psi^N(\mathcal{A}Q^N_t(x)-g(x))\big[\alpha^N \Phi^N(\nabla_\theta \mathcal{A}Q^N_t(x))\big]\cdot \nabla_\theta Q_t^N(y)d\mu(x).
\end{split}
\end{align}
We are interested in studying the limit of the dynamics of the neural network $Q^N_t$ as $N\to \infty$. Specifically, we will prove that $Q_t^N$ will converge to $Q_t$ as $N \rightarrow \infty$ where $Q_t$ satisfies  
\begin{align}\label{wide-limit}
    \begin{split}
        \frac{d Q_t}{dt}(y)=-\int_\Omega [\mathcal{A}Q_t(x)-g(x)]U(x,y)d\mu(x), \quad Q_0=0,
    \end{split}
\end{align}
where the function $U$ is
\begin{align}
U(x,y):=\mathbb{E}_{c,w,b}\bigg[\nabla_{c,w,b}\mathcal{A}[\eta(x)c\sigma(x;w,b)]\cdot \nabla_{c,w,b}[\eta(y)c\sigma(y;w,b)]\bigg],
\end{align} 
where the random variable $(c,w,b)$ has the same distribution as $(c_0^i, w_0^i, b_0^i)$.

The ODE (\ref{wide-limit}) is an infinite-dimensional linear ODE governing the evolution of the wide-limit neural network (i.e., a neural network with an ``infinite" number of hidden units) during training. The right-hand side (RHS) of the ODE involves integral over the PDE residual $[\mathcal{A}Q_t(x)-g(x)]$ weighted by a kernel $U(x,y)$. It is important to notice that the kernel $U(x,y)$ is not the standard NTK kernel: it involves the PDE operator $\mathcal{A}$, which significantly complicates its analysis. 

One of the consequences of the presence of the PDE operator $\mathcal{A}$ in the kernel $U(x,y)$ is that $U(x,y)$ is asymmetric. This is a key difference from the standard NTK kernel, which is symmetric. 

Define the integral operator $\mathcal{U}:L^2\to \mathcal{H}^2_{(0)}\subset L^2$ by \begin{align}\label{Uoperatordefn}
        \mathcal{U}f:= \int_\Omega f(x)U(x,y)d\mu(x).
    \end{align}
     Note that it is straightforward to check that $U(x,y)=0$ for $y\in \partial \Omega$ and that $U$ is $C^2_b$ with respect to $y$, which ensures that $\mathcal{U}f$ takes values in $\mathcal{H}^2_{(0)}$.
Using this notation, the limit ODE \eqref{wide-limit} can be rewritten as a linear equation in $\mathcal{H}^2_{(0)}$:
\begin{align} \label{LimitODE}
    \frac{d Q_t}{dt}=-\mathcal{U}[\mathcal{A}Q_t-g], \quad Q_0=0.
\end{align}
\begin{lemma} \label{l36}
    The ODE \eqref{LimitODE} admits a unique solution in $\mathcal{H}^2_{(0)}$. 
\end{lemma}
Now we present one of this paper's main results. The trajectory of $Q^N_t$ during training, in the limit $N \rightarrow \infty$, can be characterized by the wide-limit network $Q_t$ which satisfies the infinite-dimensional ODE \eqref{LimitODE}.
\begin{theorem} \label{wide}
For any $t \geq 0$, the neural network $Q_t^N$ converges to $Q_t$ in $\mathcal{H}^2$: \begin{align}
    \lim_{N \to \infty} \mathbb{E} [\|Q_t^N-Q_t\|_{\mathcal{H}^2}]=0.
\end{align}
\end{theorem}

\section{Analysis of the kernel function}
In order to prove global convergence as $t \rightarrow \infty$ for the limit ODE \eqref{LimitODE}, we first must prove some key properties for the integral operator. Specifically, we will study the properties of the operator $\mathcal{S}=\mathcal{A}\, \mathcal{U}$.
\begin{definition}[Operator $\mathcal{S}$]
    The operator $\mathcal{S}:L^2\to L^2$ is defined by
    \[
    \begin{split}
        \mathcal{S}f&:= \mathcal{A}\, \mathcal{U} f =\mathcal{A} \Big(\int_\Omega f(x)U(x,\cdot)d\mu(x)\Big)
    \end{split}\]
\end{definition}
\begin{definition}[Kernel $S$]
    The kernel $S$ is defined by \begin{align} \label{48}
    S(x,y):=\mathbb{E}_{c,w,b}\bigg[\nabla_{c,w,b}\mathcal{A}[\eta(x)c\sigma(x;w,b)]\cdot \nabla_{c,w,b}\mathcal{A}[\eta(y)c\sigma(y;w,b)]\bigg].
\end{align}
By symmetry of second derivatives (Clairaut--Schwarz--Young theorem) we know  $S(x,\cdot) := \mathcal{A} U(x,\cdot)$ and hence $\mathcal{S}f = \int_\Omega f(x) S(x,\cdot) d\mu(x)$.
\end{definition}
While $U(x,y)$ is asymmetric, $S(x,y)$ is symmetric. The symmetric kernel $S$ depends upon the interaction of PDE operator $\mathcal{A}$ applied to the activation function $\sigma$. It will next be proven ( Lemma \ref{ker}) that the operator $\mathcal{S}$ is discriminatory in the image set of $\mathcal{A}$. This is an important property that will later be leveraged in the global convergence proof.

\begin{lemma} \label{l43}
    The kernel $S$ is uniformly bounded. That is, there exists a constant $k>0$ such that for any $(x,y)\in \Omega^2$ we know 
        $|S(x,y)|\leq k$.
\end{lemma}

\begin{lemma} \label{lemma25}
The integral operator $\mathcal{S}$ is Hilbert--Schmidt. In addition, $\mathcal{S}$ is self-adjoint and positive semi-definite.
\end{lemma}
\begin{lemma}[Spectral decomposition] \label{l46}
    The integral operator $\mathcal{S}$ is compact, in particular, there exists an orthogonal basis of $L^2$, $\{\varepsilon_i\}_{i \in \mathbb{N}^+}\cup \{\nu_i\}_{i \in \mathbb{N}^+}$ such that \begin{align}
        \mathcal{S}\varepsilon_i=\lambda_i \varepsilon_i, \quad \mathcal{S}\nu_i=0,
    \end{align}
    where $\lambda_1 \geq \lambda_2 \geq ... >0 $.
\end{lemma}
\begin{lemma}[Projection on $\mathrm{ker}(\mathcal{S})$] \label{ker}
    For $h\in L^2$, if $\mathcal{S}h=0$ then $\langle h, \mathcal{A}f\rangle=0$ for any $f \in \mathcal{H}^2_{(0)}$.
\end{lemma}
\begin{remark}
    If $u$ is a solution to the PDE \eqref{pde1}, Lemma \ref{ker} implies that for any $t \geq 0$, the residual of our approximator $\mathcal{A}Q_t-g=\mathcal{A}[Q_t-u]$ has zero projection on the eigenfunction family $\{\nu_i\}$ since $\mathcal{S}\nu_i=0$.
\end{remark}

\begin{corollary} \label{dis}
Assuming the PDE \eqref{pde1} admits a solution $u \in \mathcal{H}^2_{(0)}$, any stationary point $Q^*\in \mathcal{H}^2_{(0)}$ of the limit ODE \eqref{wide-limit} is a solution of the PDE \eqref{pde1}. That is, any stationary point of the limit training algorithm \eqref{wide-limit} is a global minimizer and a solution of \eqref{pde1}.
\end{corollary}
\begin{remark}
   The existence of solutions is needed in this result, for example, consider the trivial case where $\mathcal{A}u \equiv 0$ and $g \neq 0$. Then  $\frac{d}{dt}Q_t =0$, but $Q_t=Q_0$ does not solve the PDE.
\end{remark}

\section{Global convergence of the limit ODE as $t \rightarrow \infty$} \label{sec5}
By analyzing the PDE residual term's projection on the eigenfunctions of $\mathcal{S}$, we can prove that the PDE residual (which is the objective function that is being minimized) converges to zero as the training time $t \rightarrow \infty$. 

Applying $\mathcal{A}$ with respect to $y$ on both sides of the limit ODE \eqref{LimitODE} yields
\begin{align} \label{59}
     \frac{d\mathcal{A}Q_t}{dt}&=\frac{d[\mathcal{A}Q_t-g]}{dt} =-\mathcal{S}[\mathcal{A}Q_t-g].
\end{align}
\begin{theorem}[Convergence of PDE residual under DGM] \label{l51}
     Assuming the PDE \eqref{pde1} admits a solution {in $\mathcal{H}^2_{(0)}$}, the PDE residual for the wide-limit neural network $Q_t$ converges to zero:
     \begin{align}
        \lim_{t \to \infty} \|\mathcal{A}Q_t-g\|_{L^2(\mu)} =0.
    \end{align}
\end{theorem}
This result establishes global convergence of the (wide-limit) training algorithm since the objective function is the PDE residual.  
However, the convergence of the PDE residual to zero does not necessarily guarantee that $\| Q_t - u \|_{L^2(\mu)}$ as $t \to \infty$ where $u$ is the solution of the PDE \eqref{pde1}. We show that, under a mild additional assumption on the PDE operator $\mathcal{A}$, we can guarantee that $Q_t$ converges to the solution of the PDE $u$.
\begin{assumption}[Bounded inverse of $\mathcal{A}$]
    The inverse of $\mathcal{A}$ is a bounded operator on $L^2 \to L^2$. That is, there exists a constant $k>0$, such that for any $g \in L^2$, the PDE \eqref{pde1} has a unique solution $u\in \mathcal{H}^2_{(0)}$ satisfying \[ \|u\|_2 \leq k \|g\|_{L^2(\mu)}. \]
    \end{assumption}
\begin{remark}
    Second-order uniformly elliptic PDEs naturally have this property. For reference, see Theorem 6, Chapter 6 of \cite{evans2022partial}.
\end{remark}
    \begin{theorem}[Convergence of $Q_t$ under DGM] \label{l62}
        If $\mathcal{A}$ has a bounded inverse, $Q_t$ converges to the  solution $u$ of the PDE: \[
        \lim_{t \to \infty} \|Q_t-u\|_{L^2(\mu)}=0.
        \]
    \end{theorem}

\section{Global Convergence of the PINN Algorithm} \label{sec7}
In many real-world physics and engineering applications, the PDE solution can be observed at a sparse set of points in the interior of the domain $\Omega$. The PINN algorithm trains a neural network model to predict the solution $u(x)$ for all $x \in \Omega$ by minimizing (via gradient descent) both the PDE residual sampled by $\mu$ on a random set of points in each epoch as well as the distance between the neural network and the observations at the finite set of sparse points. Introducing the measurement data of PDE solutions on these sparse points may help recover the PDE solution globally. 

Here, this setting is slightly different from the original PINNs formulation, as we sample points at random using the measure $\mu$. We discuss this setting because if the PDE residual is minimized only on a fixed sparse set of points, then there is no guarantee that the approximator is learning the solution -- given any function, it is possible to perturb the function locally around each measurement point to construct functions which will have zero residual at these points; such functions will generally have no relation to the solution of the PDE. For similar reasons, we assume that the boundary value of the PDE is perfectly known.

\def\B{
\begin{bmatrix}
    \mathcal{A}f-g \\
    \mathbf{f-u} \\
\end{bmatrix}}
\def\C{
\begin{bmatrix}
    \mathcal{A}Q_t-g \\
    \mathbf{Q_t-u} \\
\end{bmatrix}}
\def\D{
\begin{bmatrix}
\mathcal{S} & \Bar{\mathcal{U}}^* \\
\Bar{\mathcal{U}} & \bar{\mathcal{B}}
\end{bmatrix}
}
\def\E{
\begin{bmatrix}
    \varrho^d \\
    \varrho^p \\
\end{bmatrix}}
\def\F{
\begin{bmatrix}
    (\mathcal{U}f)(x_1) \\
    (\mathcal{U}f)(x_2) \\
    \vdots \\
    (\mathcal{U}f)(x_M)
\end{bmatrix}}
\def\G{
\begin{bmatrix}
    Q_t(x_1)-u(x_1) \\
    Q_t(x_2)-u(x_2)) \\
    \vdots \\
    Q_t(x_M)-u(x_M)
\end{bmatrix}}

Let us denote these extra observation points in $\Omega$ as $\mathbf{x}:=\{x_i\}_{i=1,2,...M}$ where the %corresponding 
observations are $u(x_i)=u_i$ for $1\leq i \leq M$. Define the measure $\mu_{\textbf{x}}=\frac{1}{M}\sum_{i=1}^M \delta_{x_i}$ and the kernel function \begin{align}
B(x,y):=\mathbb{E}_{c,w,b}\bigg[\nabla_{c,w,b}[c\eta(x)\sigma_{w,b}(x)]\cdot \nabla_{c,w,b}[c\eta(y)\sigma_{w,b}(y)]\bigg],
\end{align}
and the corresponding `integral' operator $\mathcal{B}:L^2(\mu_{\mathbf{x}})\equiv \mathbb{R}^M\to L^2(\mu)$ by \begin{align}
    \mathcal{B}\mathbf{v}(y):
    =\frac{1}{M}\sum_{i=1}^M v_i B(x_i,y).
\end{align}
For any operator $\mathcal{C}$, we define $\bar{\mathcal{C}}$ to be the evaluation of $\mathcal{C}$ on each of our training points, 
\[\bar{\mathcal{C}}v = \big[\mathcal{C}\mathbf{v}(x_1), \cdots, \mathcal{C}\mathbf{v}(x_M)\big]^\top \in \mathbb{R}^M,\]
in particular
\begin{align}\label{eqBdefn}
    \bar{\mathcal{B}}\mathbf{v}:=\Big[\frac{1}{M}\sum_{i=1}^M v_i B(x_i,x_1), \cdots ,\frac{1}{M}\sum_{i=1}^M v_i B(x_i,x_M)\Big]^\top,
\end{align}
and for generic $g\in L^2(\mu)$ and $\mathbf{v}\in \mathbb{R}^M$, direct calculation shows that 
\[\begin{split}    \langle\bar{\mathcal{U}}g,\mathbf{v}\rangle_{L^2(\mu_\mathbf{x})} &=\frac{1}{M}\sum_{i}(\mathcal{U}g)(x_i)v_i\\
&=\frac{1}{M}\sum_i\int \mathbb{E}_{c,w,b}\bigg[v_ig(x)\nabla_{c,w,b}\mathcal{A}[c\eta(x)\sigma_{w,b}(x)]\cdot \nabla_{c,w,b}[c\eta(x_i)\sigma_{w,b}(x_i)]\bigg]\mu(dx)\\
&=  \langle g, \mathcal{A}\mathcal{B}\mathbf{v}\rangle_{L^2(\mu)}, 
\end{split}\]
in other words, $\bar{\mathcal{U}}:L^2(\mu)\to \mathbb{R}^M$ is the adjoint of $\mathcal{A}\mathcal{B}:\mathbb{R}^M\to L^2(\mu)$, so we write $\mathcal{A}\mathcal{B}=\bar{\mathcal{U}}^*$.

The PINN objective function is \begin{align}
\begin{split}
    J(\theta^N)&= \|\mathcal{A}Q_t^N -g\|^2_{L^2(\mu)}+\|Q_t^N -u\|_{L^2(\mu_\mathbf{x})}^2\\&= \int_\Omega (\mathcal{A}Q_t^N-g)^2 d\mu + \int_\Omega (Q^N_t-u)^2 d\mu_{\mathbf{x}}.
\end{split}
\end{align}
Informally, the wide-limit ODE (as $N \rightarrow \infty$) for the neural network trained with continuous-time gradient descent is:
\begin{align}
    \frac{d Q_t}{dt}=-\mathcal{U}[\mathcal{A}Q_t-g]- \mathcal{B}[Q_t-u].
\end{align}
Applying the operator $\mathcal{A}$ and recalling  $\mathcal{A}\mathcal{B}$ is the adjoint of $\bar{\mathcal{U}}:L^2(\mu)\to \mathbb{R}^M$, we derive that 
 \begin{align} \label{pinns}
     \frac{d}{dt}\C=-\D \C,
 \end{align}
where, on the left-hand side of the equation, the first term $\mathcal{A}Q_t-g \in L^2(\mu)$ and the second term $\mathbf{Q_t-u} \in \mathbb{R}^M$ is \begin{align}
    \mathbf{Q_t-u}=\G.
\end{align}

In light of this, we define the operator $\mathcal{V}: L^2(\mu)\times \mathbb{R}^M \to L^2(\mu)\times \mathbb{R}^M$
\begin{align}
    \mathcal{V}f:=\D f.
\end{align}
Given the properties of $\mathcal{S}$ (Lemma \ref{lemma25}), it is easy to verify that $\mathcal{V}$ is Hilbert--Schmidt, self-adjoint and positive semi-definite. It will not generally be strictly positive definite. However, using the representation properties of neural networks, we can prove that $ (\mathcal{A}Q_t-g,  \mathbf{Q_t-u} )^{\top}$ in \eqref{pinns} has zero projection on the kernel set of $\mathcal{V}$.
\begin{lemma}[Projection on $\mathrm{ker}(\mathcal{V})$] \label{kerV}
    For $h\in L^2(\mu)\times \mathbb{R}^M$, if $\mathcal{V}h=0$ then $\langle h, \mathcal{V}m\rangle=0$ for any $m=[\mathcal{A}f-g, f(x_1)-u(x_1),\cdots, f(x_M)- u(x_M)]^\top$ where $f \in \mathcal{H}^2_{(0)}$.
\end{lemma}
\begin{theorem}[Global convergence of PINNs objective] \label{l72}
    Assume that the PDE \eqref{pde1} admits a solution. The objective function $\|\mathcal{A}Q_t-g\|^2_{L^2(\mu)}+\|Q_t-u\|_{L^2(\mu_\mathbf{x})}^2$ then converges to zero: \begin{align}
        \lim_{t \to \infty} \Big(\|\mathcal{A}Q_t-g\|^2_{L^2(\mu)}+\|Q_t-u\|_{L^2(\mu_\mathbf{x})}^2 \Big) =0.
    \end{align}
\end{theorem}
\begin{theorem}[Global convergence of PINNs] \label{l73}
    If $\mathcal{A}$ has a bounded inverse, then $Q_t$ converges to the solution of of the PDE, 
    \[
            \lim_{t \to \infty} \Big( \|Q_t-u\|^2_{L^2(\mu)}+\|Q_t-u\|_{L^2(\mu_\mathbf{x})}^2 \Big) =0.
            \]
\end{theorem}

\section{Conclusion}
In this paper, we develop a convergence theory for neural network approximators of PDEs trained with gradient descent (such as DGM and PINNs). It is proven that a neural network trained to minimize the PDE residual will converge to an  infinite-dimensional linear ODE as the number of hidden units $\rightarrow \infty$. The limit ODE's dynamics are characterized by a novel kernel function involving the PDE operator and the neural network activation function. The kernel lacks a spectral gap, making the analysis of the limit ODE challenging. Using an eigendecomposition approach, we are able to prove that the PDE residual of the limit neural network converges to zero. Furthermore, under mild assumptions, the limit neural network converges to the solution of the PDE.

\newpage
\appendix
\section{Appendix} \label{App}
\subsection{Proof of Lemma \ref{l35}}
\begin{proof}
    Notice that in \eqref{9}, three terms $\alpha^N$, $\psi^N(\mathcal{A}Q^N-g)$ and $\Phi^N(\nabla_\theta \mathcal{A}Q^N)$ are bounded by $N^{2\beta-1}$, $N^\delta$ and $N^{\epsilon-\beta}$ respectively.
\end{proof}
\subsection{Proof of Lemma \ref{l36}}
\begin{proof}
    By Assumptions \ref{assume_1}, \ref{assume_activation} and \ref{initialization}, the operator $\mathcal{U}$ is Lipschitz in $\mathcal{H}^2$ norm. Therefore it admits a unique solution in $\mathcal{H}^2$. It is easy to verify that, as its initial condition is in $\mathcal{H}^2_{(0)}$, and $\mathcal{U}$ has codomain $\mathcal{H}^2_{(0)}$, the solution lives in the Hilbert subspace $\mathcal{H}^2_{(0)}$.
\end{proof}
\subsection{Proof of Theorem \ref{wide}}
\begin{proof}
The main idea of the proof is the split the difference into multiple residual terms and provide a bound for each of them. Then we apply Gr\"onwall's inequality to provide an estimate of the difference between $Q^N$ and $Q$. In this proof, constant $C$ may vary from line to line, but it remains invariant of $N$ and $t$.

To simplify our notations, we set $\mathcal{L}Q:=\mathcal{A}Q-g$. In the integral form, we have \begin{align} \label{qtn}
        Q_t^N(y)=Q_0^N(y)-\int_0^t \int_\Omega \psi^N(\mathcal{L}Q^N_t(x))\alpha^N \Phi^N(\nabla_\theta \mathcal{A}Q^N_s(x))\cdot \nabla_\theta Q_s^N(y) d\mu(x) ds.
    \end{align}
    Similarly, at the same time \begin{align} \label{qt}
        Q_t(y)=Q_0(y)-\int_0^t\int_\Omega \mathcal{L}Q_t(x)U(x,y)d\mu(x) ds.
    \end{align}
    Subtracting \eqref{qt} from \eqref{qtn} and taking partial derivative with respect to $y$ with indices $\alpha$ give \begin{align}
        \begin{split}
            &|D_\alpha(Q_t^N-Q_t)(y)|\\
            &\leq \int_0^t \int_\Omega \bigg|\psi^N(\mathcal{L}Q^N_s(x))\alpha^N \Phi^N(\nabla_\theta \mathcal{A}Q^N_s(x))\cdot D_\alpha \nabla_\theta Q_s^N(y)-\mathcal{L}Q^N_t(x)D_\alpha^y U(x,y)\bigg | d\mu(x)ds\\
            & \quad +|D_\alpha(Q_0^N-Q_0)(y)|\\
            &\leq \int_0^t \int_\Omega \bigg|\psi^N(\mathcal{L}Q^N_s(x))\alpha^N \Phi^N(\nabla_\theta \mathcal{A}Q^N_s(x))\cdot D_\alpha (\nabla_\theta Q_s^N-\nabla_\theta Q_0^N)(y) \bigg|d\mu(x)ds\\
            &\quad + \int_0^t \int_\Omega \bigg|\psi^N(\mathcal{L}Q^N_s(x))\alpha^N (\Phi^N(\nabla_\theta \mathcal{A}Q^N_s(x))-\Phi^N(\nabla_\theta \mathcal{A}Q^N_0(x))\cdot D_\alpha \nabla_\theta Q_0^N(y) \bigg|d\mu(x)ds\\
            &\quad + \int_0^t \int_\Omega \bigg|\psi^N(\mathcal{L}Q^N_s(x))\alpha^N (\Phi^N(\nabla_\theta \mathcal{A}Q^N_0(x))-\nabla_\theta \mathcal{A}Q^N_0(x))\cdot D_\alpha \nabla_\theta Q_0^N(y) \bigg|d\mu(x)ds\\
            &\quad + \int_0^t \int_\Omega \bigg|\psi^N(\mathcal{L}Q^N_s(x)) \big[\alpha^N\nabla_\theta \mathcal{A}Q^N_0(x)\cdot D_\alpha \nabla_\theta Q_0^N(y)-D_\alpha^y U(x,y)\big] \bigg|d\mu(x)ds\\
            &\quad + \int_0^t \int_\Omega \bigg|\big[\psi^N(\mathcal{L}Q^N_s(x))-\psi^N(\mathcal{L}Q_s(x))\big]D_\alpha^y U(x,y)) \bigg|d\mu(x)ds\\
            &\quad + \int_0^t \int_\Omega \bigg|\big[\psi^N(\mathcal{L}Q_s(x))-\mathcal{L}Q_s(x))\big]D_\alpha^y U(x,y)) \bigg|d\mu(x)ds + |D_\alpha(Q_0^N-Q_0)(y)|.
        \end{split}
    \end{align}
    Now by the fact that $|\psi^N|<N^\delta$, $|D_\alpha^y U(x,y)|<C$ we have \begin{align} \label{16}
        \begin{split}
            &|D_\alpha(Q_t^N-Q_t)(y)|\\
            &\leq \int_0^t \int_\Omega N^{2\beta+\delta-1} \bigg|\Phi^N(\nabla_\theta \mathcal{A}Q^N_s(x))\cdot D_\alpha (\nabla_\theta Q_s^N-\nabla_\theta Q_0^N)(y) \bigg|d\mu(x)ds\\
            &\quad + \int_0^t \int_\Omega N^{2\beta+\delta-1} \bigg|(\Phi^N(\nabla_\theta \mathcal{A}Q^N_s(x))-\Phi^N(\nabla_\theta \mathcal{A}Q^N_0(x))\cdot D_\alpha \nabla_\theta Q_0^N(y) \bigg|d\mu(x)ds\\
            &\quad + \int_0^t \int_\Omega N^{2\beta+\delta-1} \bigg|(\Phi^N(\nabla_\theta \mathcal{A}Q^N_0(x))-\nabla_\theta \mathcal{A}Q^N_0(x))\cdot D_\alpha \nabla_\theta Q_0^N(y) \bigg|d\mu(x)ds\\
            &\quad + \int_0^t \int_\Omega N^{\delta} \bigg|N^{2\beta-1}\nabla_\theta \mathcal{A}Q^N_0(x)\cdot D_\alpha \nabla_\theta Q_0^N(y)-D_\alpha^y U(x,y) \bigg|d\mu(x)ds\\
            &\quad + C\int_0^t \int_\Omega \bigg|\psi^N(\mathcal{L}Q^N_s(x))-\psi^N(\mathcal{L}Q_s(x)) \bigg|d\mu(x)ds\\
            &\quad + C\int_0^t \int_\Omega \bigg|\psi^N(\mathcal{L}Q_s(x))-\mathcal{L}Q_s(x)) \bigg|d\mu(x)ds + |D_\alpha(Q_0^N-Q_0)(y)|.
        \end{split}
    \end{align}
    Let us denote $V_t^N(y):=\sum_{|\alpha|\leq 2} |D_\alpha (Q_t^N-Q_t)(y)|$. Then by summing inequality \eqref{16} with respect to all indices and integrating with respect to $\mu(y)$ \begin{align} \label{17}
        \begin{split}
            \int_\Omega V_t^N(y)^2 d\mu(y) \leq  \bigg(\int_\Omega V_t^N(y)^2 d\mu(y)\bigg)^{\frac{1}{2}}\bigg[C\int_0^t \big(\int_\Omega V_s^N(x)^2 d\mu(x)\big)^{\frac{1}{2}}ds + M \bigg]
        \end{split}
    \end{align}
    where $M:=M_1+M_2+M_3+M_4+M_5+M_6$ denotes the residual terms \begin{align}
        \begin{split}
            M_1:&=\int_0^T N^{2\beta+\delta-1} \bigg(\int_{\Omega^2}\big| \Phi^N(\nabla_\theta \mathcal{A}Q^N_s(x))\cdot \sum_\alpha D_\alpha (\nabla_\theta Q_s^N-\nabla_\theta Q_0^N)(y) \big|^2 d\mu(x)d\mu(y)\bigg)^{\frac{1}{2}}ds,\\
            M_2:&=\int_0^T N^{2\beta+\delta-1} \bigg(\int_{\Omega^2}\big|(\Phi^N(\nabla_\theta \mathcal{A}Q^N_s(x))-\Phi^N(\nabla_\theta \mathcal{A}Q^N_0(x)))\cdot \sum_\alpha D_\alpha \nabla_\theta Q_0^N(y) \big|^2 d\mu(x)d\mu(y)\bigg)^{\frac{1}{2}}ds,\\
            M_3:&=\int_0^T N^{2\beta+\delta-1} \bigg(\int_{\Omega^2}\big|(\Phi^N(\nabla_\theta \mathcal{A}Q^N_0(x))-\nabla_\theta \mathcal{A}Q^N_0(x))\cdot \sum_\alpha D_\alpha \nabla_\theta Q_0^N(y) \big|^2 d\mu(x)d\mu(y)\bigg)^{\frac{1}{2}}ds,\\
            M_4:&=\int_0^T N^{\delta} \bigg(\int_{\Omega^2}\big|N^{2\beta-1}\nabla_\theta \mathcal{A}Q^N_0(x)\cdot \sum_\alpha D_\alpha \nabla_\theta Q_0^N(y)-\sum_\alpha D_\alpha^y U(x,y) \big|^2 d\mu(x)d\mu(y)\bigg)^{\frac{1}{2}}ds,\\
            M_5:&=\int_0^T \int_{\Omega}\big|\psi^N(\mathcal{L}Q_s(x))-\mathcal{L}Q_s(x) \big| d\mu(x)ds,\\
            M_6:&=\bigg(\int_\Omega \big[\sum_\alpha |D_\alpha(Q_0^N-Q_0)(y)|\big]^2 d\mu(y)\bigg)^{\frac{1}{2}}.
        \end{split}
    \end{align}
    Then by Gr\"onwall's inequality, since from \eqref{17} we have \begin{align} \label{gronwall}
        \bigg(\int_\Omega V_t^N(y)^2 d\mu(y)\bigg)^{\frac{1}{2}} \leq C \int_0^t \bigg(\int_\Omega V_s^N(x)^2 d\mu(x)\bigg)^{\frac{1}{2}}ds + M,
    \end{align}
    we derive that \begin{align}
        \int_0^t \bigg(\int_\Omega V_s^N(x)^2 d\mu(x)\bigg)^{\frac{1}{2}}ds \leq \frac{M}{C}e^{Ct}.
    \end{align}
    Taking expectation on both sides, by Lemmas \ref{M1}, \ref{M2}, \ref{M3}, \ref{M4}, \ref{M5}, \ref{M6} we can control each of the terms in $M$. Hence \begin{align}
        \lim_{N \to \infty} \mathbb{E}\bigg[\int_0^t \bigg(\int_\Omega V_s^N(x)^2 d\mu(x)\bigg)^{\frac{1}{2}}ds \bigg] =0.
    \end{align}
    By \eqref{gronwall} we have for any $0 \leq s \leq T$ \begin{align}
        \lim_{N \to \infty}\mathbb{E}\bigg[\bigg(\int_\Omega V_s^N(x)^2 d\mu(x)\bigg)^{\frac{1}{2}}\bigg]=0.
    \end{align}
    Notice that $\int_\Omega V_s^N(x)^2 d\mu(x) \geq \|Q_s^N-Q_s\|_{\mathcal{H}^2}$, we conclude that for any $0 \leq s \leq T$ \begin{align}
        \lim_{N \to \infty} \mathbb{E}\bigg[\|Q_s^N-Q_s\|_{\mathcal{H}^2}\bigg]=0.
    \end{align}
\end{proof}
\subsection{Lemmas \ref{M1} to \ref{M6}}
\begin{lemma} \label{M1}
    Residual term $M_1$ satisfies \begin{align}
         \lim_{N \to \infty}\mathbb{E}\bigg[ \int_0^T N^{2\beta+\delta-1} \bigg(\int_{\Omega^2}\big| \Phi^N(\nabla_\theta \mathcal{A}Q^N_s(x))\cdot \sum_\alpha D_\alpha (\nabla_\theta Q_s^N-\nabla_\theta Q_0^N)(y) \big|^2 d\mu(x)d\mu(y)\bigg)^{\frac{1}{2}}ds \bigg] = 0.
    \end{align}
\end{lemma}
\begin{proof}
    By Jensen's inequality, it suffices to prove that \begin{align}
         \int_0^T N^{2\beta+\delta-1} \bigg(\int_{\Omega^2}\mathbb{E}\bigg[\big| \Phi^N(\nabla_\theta \mathcal{A}Q^N_s(x))\cdot \sum_\alpha D_\alpha (\nabla_\theta Q_s^N-\nabla_\theta Q_0^N)(y) \big|^2 \bigg] d\mu(x)d\mu(y)\bigg)^{\frac{1}{2}}ds \to 0.
    \end{align}
    It suffices to show that for any indices $\alpha$ \begin{align} \label{21}
        \int_0^T N^{2\beta+\delta-1} \bigg(\int_{\Omega^2}\mathbb{E}\bigg[\big| \Phi^N(\nabla_\theta \mathcal{A}Q^N_s(x))\cdot D_\alpha (\nabla_\theta Q_s^N-\nabla_\theta Q_0^N)(y) \big|^2 \bigg] d\mu(x)d\mu(y)\bigg)^{\frac{1}{2}}ds \to 0.
    \end{align}
    Notice that $|\Phi^N|$ is bounded elementwise by $N^{\epsilon-\beta}$. We also notice that by the mean value theorem and the fact that $\eta$, $\sigma$, and their derivatives are up to polynomial growth  \begin{align}
    \begin{split}
        \|D_\alpha^y (\nabla_\theta Q_s^N-\nabla_\theta Q_0^N)(y)\|_1\leq \frac{C}{N^{\beta}}\sum_{i=1}^N \bigg( (\|w_s^i-w_0^i\|+\|b_s^i-b_0^i\|+\|c_s^i-c_0^i\|)f(w_0^i,c_0^i,b_0^i,y)\bigg)
    \end{split}
    \end{align}
    where $f$ is a polynomial up to 3rd order. We notice that $\|w_u^i-w_0^i\|+\|b_u^i-b_0^i\|+\|c_u^i-c_0^i\|\leq u C N^{\epsilon+\delta+\beta-1}$. Therefore \begin{align}
        \mathbb{E}\bigg[\big| \Phi^N(\nabla_\theta \mathcal{A}Q^N_s(x))\cdot D_\alpha (\nabla_\theta Q_s^N-\nabla_\theta Q_0^N)(y) \big|^2 \bigg] \leq C N^{2(2\epsilon+\delta-\beta)}.
    \end{align}
    Therefore, the left-hand side of \eqref{21} is smaller than $C N^{2\delta+2\epsilon+\beta-1}$, which goes to 0 as $N \to \infty$.
\end{proof}
\begin{lemma} \label{M2}
    Residual term $M_2$ satisfies \begin{align}
         \lim_{N \to \infty}\mathbb{E}\bigg[ \int_0^T N^{2\beta+\delta-1} \bigg(\int_{\Omega^2}\big|(\Phi^N(\nabla_\theta \mathcal{A}Q^N_s(x))-\Phi^N(\nabla_\theta \mathcal{A}Q^N_0(x)))\cdot \sum_\alpha D_\alpha \nabla_\theta Q_0^N(y) \big|^2 d\mu(x)d\mu(y)\bigg)^{\frac{1}{2}}ds \bigg] = 0.
    \end{align}
\end{lemma}
\begin{proof}
    It suffices to show that for any $\alpha$ \begin{align} \label{25}
        \int_0^T N^{2\beta+\delta-1} \bigg(\int_{\Omega^2}\mathbb{E}\bigg[\big|(\Phi^N(\nabla_\theta \mathcal{A}Q^N_s(x))-\Phi^N(\nabla_\theta \mathcal{A}Q^N_0(x)))\cdot D_\alpha \nabla_\theta Q_0^N(y) \big|^2 \bigg] d\mu(x)d\mu(y)\bigg)^{\frac{1}{2}}ds \to 0.
    \end{align}
    As, elementwise, $|\phi^N(x)-\phi^N(y)|\leq |x-y|$, we have \begin{align}
        \big|(\Phi^N(\nabla_\theta \mathcal{A}Q^N_s(x))-\Phi^N(\nabla_\theta \mathcal{A}Q^N_0(x))\cdot D_\alpha \nabla_\theta Q_0^N(y) \big| \leq \big|\nabla_\theta \mathcal{A}Q^N_s(x)-\nabla_\theta \mathcal{A}Q^N_0(x)\big |\cdot \big |D_\alpha \nabla_\theta Q_0^N(y) \big|
    \end{align}
    Then, since \begin{align}
    \begin{split}
        \big|(\nabla_\theta \mathcal{A}Q^N_s(x)-\nabla_\theta \mathcal{A}Q^N_0(x)\big |\cdot \big| D_\alpha \nabla_\theta Q_0^N(y) \big| \leq \frac{C}{N^{2\beta}}\sum_{i=1}^N \bigg( (\|w_s^i-w_0^i\|+\|b_s^i-b_0^i\|+\|c_s^i-c_0^i\|)g(w_0^i,c_0^i,b_0^i,x,y)\bigg)
    \end{split}
    \end{align}
    where $g$ is a polynomial up to 3rd order. Since $\|w_u^i-w_0^i\|+\|b_u^i-b_0^i\|+\|c_u^i-c_0^i\|\leq u C N^{\epsilon+\delta+\beta-1}$, we have \begin{align}
        \mathbb{E}\bigg[\big|(\Phi^N(\nabla_\theta \mathcal{A}Q^N_s(x))-\Phi^N(\nabla_\theta \mathcal{A}Q^N_0(x)))\cdot D_\alpha \nabla_\theta Q_0^N(y) \big|^2 \bigg] \leq CN^{2(\epsilon+\delta-\beta)}.
    \end{align}
    The left-hand side of \eqref{25} is bounded by $C N^{\epsilon+\delta+\beta-1}$, which goes to $0$ as $N$ goes to infinity.
\end{proof}
\begin{lemma}\label{M3}
    Residual term $M_3$ satisfies \begin{align}
        \begin{split}
            \lim_{N \to \infty}\mathbb{E}\bigg[\int_0^T N^{2\beta+\delta-1} \bigg(\int_{\Omega^2}\big|(\Phi^N(\nabla_\theta \mathcal{A}Q^N_0(x))-\nabla_\theta \mathcal{A}Q^N_0(x))\cdot \sum_\alpha D_\alpha \nabla_\theta Q_0^N(y) \big|^2 d\mu(x)d\mu(y)\bigg)^{\frac{1}{2}}ds \bigg] = 0.
        \end{split}
    \end{align}
\end{lemma}
\begin{proof}
    By Jensen's inequality, it suffices to show that \begin{align}
        \int_0^T N^{2\beta+\delta-1} \bigg(\int_{\Omega^2}\mathbb{E}\bigg[\big|(\Phi^N(\nabla_\theta \mathcal{A}Q^N_0(x))-\nabla_\theta \mathcal{A}Q^N_0(x))\cdot D_\alpha \nabla_\theta Q_0^N(y) \big|^2\bigg] d\mu(x)d\mu(y)\bigg)^{\frac{1}{2}}ds \to 0
    \end{align} for any indices $\alpha$ where $|\alpha|\leq 2$.
    
    Let us denote the $k$-th unit of the initial approximator $Q^N_0$ as $q^k$, where $q^k=c_0^kN^{-\beta}\eta\sigma_{w_0^k,b_0^k}$.
    Then \begin{align} \begin{split}
        &\mathbb{E}\bigg[\bigg|\big(\Phi^N(\nabla_\theta \mathcal{A}Q^N_0(x)\big)-\nabla_\theta \mathcal{A}Q^N_0(x))\cdot D_\alpha \nabla_\theta Q_0^N(y) \bigg|^2\bigg]\\
        &= \mathbb{E}\bigg[ \bigg(\sum_{k=1}^N [\Phi^N(\nabla_\theta \mathcal{A}q^k(x))-\nabla_\theta \mathcal{A}q^k(x)]\cdot \nabla_\theta D_\alpha q^k(y) \bigg)^2 \bigg]\\
        &=N\mathbb{E}\bigg[\big(\Phi^N(\nabla_\theta \mathcal{A}q(x))-\nabla_\theta \mathcal{A}q(x)]\cdot \nabla_\theta D_\alpha q(y)\big)^2\bigg]+N(N-1)\mathbb{E}\bigg[\Phi^N(\nabla_\theta \mathcal{A}q(x))-\nabla_\theta \mathcal{A}q(x)]\cdot \nabla_\theta D_\alpha q(y)\bigg]^2
    \end{split}
    \end{align}
    Then by definition \begin{align}
        \begin{split}
            &\Phi^N(\nabla_\theta \mathcal{A}q(x))-\nabla_\theta \mathcal{A}q(x)]\cdot \nabla_\theta D_\alpha q(y)\\
            &= [\phi^N(\mathcal{A}[N^{-\beta}\eta \sigma_{w,b}](x))-\mathcal{A}[N^{-\beta}\eta \sigma_{w,b}](x)]\cdot D_\alpha [N^{-\beta}\eta \sigma_{w,b}](y)] \\
            &\quad + \sum_{i=1}^d [\phi^N(\mathcal{A}[N^{-\beta}c x_i \eta \sigma_{w,b}'](x))-\mathcal{A}[N^{-\beta}c x_i \eta \sigma_{w,b}'](x)]\cdot D_\alpha [N^{-\beta}c y_i \eta \sigma_{w,b}'](y)]\\
            &\quad + [\phi^N(\mathcal{A}[N^{-\beta}c \eta \sigma_{w,b}'](x))-\mathcal{A}[N^{-\beta}c \eta \sigma_{w,b}'](x)]\cdot D_\alpha [N^{-\beta}c  \eta \sigma_{w,b}'](y)].
        \end{split}
    \end{align}
    We introduce a uniform bound for $|D_\alpha[\eta \sigma_{w,b}(x)]|$, $|D_\alpha[c x_i \eta \sigma_{w,b}(x)]|$ and $|D_\alpha[c\eta \sigma_{w,b}(x)]|$ for any indices $\alpha$ and any $x$. By the fact that $\mathcal{A}$ is Lipschitz \begin{align}
    \begin{split}
        |\mathcal{A}[\eta \sigma_{w,b}](x)| &\leq k_0 \sum_{0 \leq |\alpha| \leq 2} |D_\alpha[\eta \sigma_{w,b}](x)|.
    \end{split}
    \end{align}
    Notice that \begin{align}
        \begin{split}
            |D_\alpha[\eta \sigma_{w,b}](x)| &= \bigg|\sum_{\alpha_1+\alpha_2=\alpha} D_{\alpha_1}\eta \cdot D_{\alpha_2}\sigma_{w,b}(x)\bigg|\\
            & \leq k_\eta \bigg| \sum_{\alpha_1+\alpha_2=\alpha} D_{\alpha_2}\sigma_{w,b}(x)\bigg|\\
            & \leq k_\eta \bigg| \sum_{0 \leq|\alpha_2|\leq 2} D_{\alpha_2}\sigma_{w,b}(x)\bigg|\\
            & \leq k_\eta k_\sigma k_\Omega \bigg( 1+\sum_{i=1}^d |(w)_i|+\sum_{i,j=1}^d |(w)_i(w)_j|\bigg).
        \end{split}
    \end{align}
    Therefore \begin{align}
        |\mathcal{A}[\eta \sigma_{w,b}](x)| \leq k_0 \sum_{0 \leq |\alpha|\leq 2}|D_\alpha[\eta \sigma_{w,b}](x)| \leq k \bigg(1+\sum_{i=1}^d |(w)_i|+\sum_{i,j=1}^d |(w)_i(w)_j|\bigg).
    \end{align}
    Similar results hold for $|D_\alpha[c x_i \eta \sigma_{w,b}]|$ and $|D_\alpha[c\eta \sigma_{w,b}]|$. We define $f(w):=k (1+\sum_{i=1}^d |(w)_i|+\sum_{i,j=1}^d |(w)_i(w)_j|)$. Now \begin{align}
        \begin{split}
            |\phi^N(\mathcal{A}[N^{-\beta}\eta \sigma_{w,b}](x))-\mathcal{A}[N^{-\beta}\eta \sigma_{w,b}](x)|&\leq N^{-\beta} \bigg[|\mathcal{A}[\eta \sigma_{w,b}](x)|-N^\epsilon \bigg]\mathbf{1}_{\{|\mathcal{A}[\eta \sigma_{w,b}](x)|\geq N^\epsilon\}}(x)\\
            &\leq N^{-\beta}\big[f(w)-N^\epsilon \big]\mathbf{1}_{\{f(w)\geq N^\epsilon\}}(x).
        \end{split}
    \end{align}
    Subsequently, \begin{align}
        |\Phi^N(\nabla_\theta \mathcal{A}q(x))-\nabla_\theta \mathcal{A}q(x)]\cdot \nabla_\theta D_\alpha q(y)|\leq N^{-2\beta}(d+2)\big[f(w)-N^\epsilon \big]\mathbf{1}_{\{f(w)\geq N^\epsilon\}}(x).
    \end{align}
    Therefore \begin{align} \label{k1}
        \begin{split}
            N\mathbb{E}\bigg[\big(\Phi^N(\nabla_\theta \mathcal{A}q(x))-\nabla_\theta \mathcal{A}q(x)]\cdot \nabla_\theta D_\alpha q(y)\big)^2\bigg] \leq N^{1-4\beta}(d+2)\mathbb{E}[f(w)^2] \leq kN^{1-4\beta}.
        \end{split}
    \end{align}
    Meanwhile \begin{align}
        \begin{split}
            \mathbb{E}\bigg[\Phi^N(\nabla_\theta \mathcal{A}q(x))-\nabla_\theta \mathcal{A}q(x)]\cdot \nabla_\theta D_\alpha q(y)\bigg]&\leq \mathbb{E}[N^{-2\beta}(d+2)\big[f(w)-N^\epsilon \big]\mathbf{1}_{\{f(w)\geq N^\epsilon\}}(x)]\\
            & \leq N^{-2\beta}(d+2)\mathbb{E}[f(w) \mathbf{1}_{\{f(w)\geq N^\epsilon\}}(x)]\\
            & \leq k_0 N^{-2\beta-\epsilon}\mathbb{E}[f(w)^2] \leq k N^{-2\beta-\epsilon}.
        \end{split}
    \end{align}
    Therefore
    \begin{align} \label{k2}
        \begin{split}
            N(N-1)\mathbb{E}\bigg[\Phi^N(\nabla_\theta \mathcal{A}q(x))-\nabla_\theta \mathcal{A}q(x)]\cdot \nabla_\theta D_\alpha q(y)\bigg]^2 \leq k N^{2-4\beta-2\epsilon}.
        \end{split}
    \end{align}
    Combining \eqref{k1} and \eqref{k2}, we conclude that \begin{align}
        M_4 \leq k (N^{1/2-\delta}+ N^{\delta-\epsilon})\to 0
    \end{align}
    as $N \to \infty$.
\end{proof}
\begin{lemma} \label{M4}
    Residual term $M_4$ satisfies \begin{align}
        \begin{split}
            \lim_{N \to \infty}\mathbb{E}\bigg[\int_0^T N^{\delta} \bigg(\int_{\Omega^2}\big|N^{2\beta-1}\nabla_\theta \mathcal{A}Q^N_0(x)\cdot \sum_\alpha D_\alpha \nabla_\theta Q_0^N(y)-\sum_\alpha D_\alpha^y U(x,y) \big|^2 d\mu(x)d\mu(y)\bigg)^{\frac{1}{2}}ds\bigg] = 0.
        \end{split}
    \end{align}
\end{lemma}
\begin{proof}
    It suffices to prove that for any $\alpha$ \begin{align} \label{31}
        \int_0^T N^{\delta} \bigg(\int_{\Omega^2}\mathbb{E}\bigg[\big|N^{2\beta-1}\nabla_\theta \mathcal{A}Q^N_0(x)\cdot  D_\alpha \nabla_\theta Q_0^N(y)-D_\alpha^y U(x,y) \big|^2 \bigg]d\mu(x)d\mu(y)\bigg)^{\frac{1}{2}}ds \to 0.
    \end{align}
    By the strong law of large numbers: \begin{align}
        \lim_{N \to \infty} N^{2\beta-1}\nabla_\theta \mathcal{A}Q^N_0(x)\cdot D_\alpha \nabla_\theta Q_0^N(y)=D_\alpha^y U(x,y).
    \end{align}
    Therefore \begin{align}
    \begin{split}
        \mathbb{E}\bigg[\big|N^{2\beta-1}\nabla_\theta \mathcal{A}Q^N_0(x)\cdot  D_\alpha \nabla_\theta Q_0^N(y)-D_\alpha^y U(x,y) \big|^2 \bigg]&=\text{Var}\big[N^{2\beta-1}\nabla_\theta \mathcal{A}Q^N_0(x)\cdot  D_\alpha \nabla_\theta Q_0^N(y)\big]\\
        &= \frac{1}{N}\text{Var}[D_\alpha^y U_{c,w,b}(x,y)].
    \end{split}
    \end{align}
    Consequently, the left-hand side of \eqref{31} is bounded by $CN^{\delta-1}$ which vanishes as $N$ goes to infinity.
\end{proof}
\begin{lemma} \label{M5}
    Residual term $M_5$ satisfies \begin{align}
        \lim_{N \to \infty} \int_0^T \int_{\Omega}\big|\psi^N(\mathcal{L}Q_s(x))-\mathcal{L}Q_s(x) \big| d\mu(x)ds =0.
    \end{align}
\end{lemma}
\begin{proof}
    Notice that \begin{align}
        |\psi^N(\mathcal{L}Q_s(x))-\mathcal{L}Q_s(x) \big| \leq |\mathcal{L}Q_s(x)|\mathbf{1}_{\{|\mathcal{L}Q_s|> N^\delta\}}(x).
    \end{align}
    By the dominated convergence theorem, we conclude our proof.
\end{proof}
\begin{lemma} \label{M6}
    Residual term $M_6$ satisfies \begin{align}
        \begin{split}
            \lim_{N \to \infty}\mathbb{E}\bigg[\bigg(\int_\Omega \big[\sum_\alpha |D_\alpha(Q_0^N-Q_0)(y)|\big]^2 d\mu(y)\bigg)^{\frac{1}{2}}\bigg] = 0.
        \end{split}
    \end{align}
\end{lemma}
\begin{proof}
    By Jensen's inequality, it suffices to show that for any $\alpha$ \begin{align}
        \int_\Omega \mathbb{E}\big[D_\alpha(Q_0^N-Q_0)(y)^2\big] d\mu(y) \to 0.
    \end{align}
    Since $\mathbb{E}[Q_0^N]=Q_0=0$ holds for all $y$ we have \begin{align}
        \mathbb{E}\big[D_\alpha(Q_0^N-Q_0)(y)^2\big]=\text{Var}[D_\alpha Q_0^N(y)]=N^{1-2\beta}\text{Var}[D_\alpha c\sigma(wy+b)]\leq C N^{1-2\beta}.
    \end{align}
    Integrating with respect to $\mu(y)$, and letting $N$ go to infinity finish the proof.
\end{proof}

\subsection{Proof of Lemma \ref{l43}}
\begin{proof}
    By definition \begin{align}
        \begin{split}
            S(x,y)&=\mathbb{E}_{c,w,b}\bigg[ \mathcal{A}[\eta(x)\sigma_{w,b}(x)]\mathcal{A}[\eta(y)\sigma_{w,b}(y)]+\sum_{i=1}^d \mathcal{A}[c\eta(x)x_i\sigma_{w,b}'(x)]\mathcal{A}[c\eta(y)y_i\sigma_{w,b}'(y)]\\
            &\quad + \mathcal{A}[c\eta(x)\sigma_{w,b}'(x)]\mathcal{A}[c\eta(y)\sigma_{w,b}'(y)]\bigg].
        \end{split}
    \end{align}
    By the fact that $\mathcal{A}$ is Lipschitz, and that $\eta$, $\sigma_{w,b}$ and their partial derivatives are all bounded, there exists constant $k_1>0$, $k_2>0$ such that\begin{align}
        \begin{split}
            |\mathcal{A}[\eta(x)\sigma_{w,b}(x)]|\leq k\sum_{\alpha}|D_\alpha \eta(x)\sigma_{w,b}(x)| \leq k_1 \sum_{1\leq i,j \leq d} |w_i|+|w_iw_j|+1.
        \end{split}
    \end{align} 
    Similarly, since $c$ is bounded and $\Omega$ is bounded \begin{align}
        \begin{split}
            |\mathcal{A}[c\eta(x)\sigma_{w,b}'(x)]|&\leq k_1 \sum_{1\leq i,j \leq d} |w_i|+|w_iw_j|+1\\
            |\mathcal{A}[c\eta(x)x_i\sigma_{w,b}'(x)]|&\leq  k_1 \sum_{1\leq i,j \leq d} |w_i|+|w_iw_j|+1.
        \end{split}
    \end{align}
    Therefore,
    \begin{align}
        \begin{split}
            |S(x,y)|\leq \mathbb{E}_{c,w,b}\big[(d+2)k_1^2 (\sum_{1\leq i,j \leq d} |w_i|+|w_iw_j|+1)^2\big]\leq k_2.
        \end{split}
    \end{align}
\end{proof}
\subsection{Proof of Lemma \ref{lemma25}}
\begin{proof}
    Since $S(x,y)$ is uniformly bounded, $\iint_{\Omega^2}|S(x,y)|^2d\mu(x)d\mu(y)< \infty$. Therefore $\mathcal{S}$ is Hilbert--Schmidt. And since $S(x,y)$ is symmetric, the operator $\mathcal{S}$ is self-adjoint with respect to the $L^2$ inner product. Now it remains to prove that $\mathcal{S}$ is positive semi-definite.
    For $f\in L^2$ \begin{align}
        \begin{split}
            \langle f, \mathcal{S}f \rangle &= \iint_{\Omega^2} f(x)S(x,y)f(y)d \mu(x)d \mu(y)\\&= \iint_{\Omega^2}f(x)\mathbb{E}_{c,w,b}\big[\nabla_{c,w,b}\mathcal{A}[\eta(x)c\sigma(x;w,b)]\cdot \nabla_{c,w,b}\mathcal{A}[\eta(y)c\sigma(y;w,b)]\big] f(y)d\mu(x)d\mu(y)
        \end{split}
    \end{align}
    By Tonelli's theorem, swapping the order of expectation and the integral gives \begin{align} \label{40}
        \begin{split}
            \langle f, \mathcal{S}f \rangle &= \mathbb{E}_{c,w,b}\bigg[ \iint_{\Omega^2}f(x)\nabla_{c,w,b}\mathcal{A}[\eta(x)c\sigma(x;w,b)]\cdot \nabla_{c,w,b}\mathcal{A}[\eta(y)c\sigma(y;w,b)] f(y)d\mu(x)d\mu(y) \bigg] \\
            & \geq \mathbb{E}_{c,w,b}\bigg[ \iint_{\Omega^2}f(x)\nabla_{c}\mathcal{A}[\eta(x)c\sigma(x;w,b)]\cdot \nabla_{c}\mathcal{A}[\eta(y)c\sigma(y;w,b)] f(y)d\mu(x)d\mu(y) \bigg]\\
            & = \mathbb{E}_{c,w,b}\bigg[ \iint_{\Omega^2}f(x)\mathcal{A}[\eta(x)\sigma(x;w,b)]\mathcal{A}[\eta(y)\sigma(y;w,b)] f(y)d\mu(x)d\mu(y) \bigg]\\
            &= \mathbb{E}_{c,w,b}\bigg[ \big(\int_{\Omega}f(x)\mathcal{A}[\eta(x)\sigma(x;w,b)]d\mu(x)\big)^2 \bigg] \geq 0.
        \end{split}
    \end{align}
\end{proof}
\subsection{Proof of Lemma \ref{l46}}
\begin{proof}
    As $\mathcal{S}$ is a Hilbert--Schmidt integral operator, $\mathcal{S}$ is compact. Since $\mathcal{S}$ is self-adjoint, the spectral theorem applies. From Lemma \ref{lemma25}, we see that $\mathcal{S}$ is positive semi-definite. Thus, its eigenvalues are real, non-negative, and concentrate only at zero. We use $\{\nu_i\}$ to represent the eigenfunctions of the zero eigenvalue, and $\{\varepsilon_i\}$ for eigenfunctions of positive eigenvalues.
\end{proof}
\subsection{Proof of Lemma \ref{ker}}
We recall without proof the following technical result.
\begin{lemma}[Lemma 5 in \cite{cohen2022neural}] \label{lemma_eta} 
Given Assumptions \ref{a1} and \ref{a2}: \begin{enumerate}%[label=\roman*.]
    \item The set of functions $C^3(\overline\Omega)\cap C_0(\overline\Omega)$ is dense in $\mathcal{H}^2_{(0)} = \mathcal{H}^2\cap \mathcal{H}^1_0$ (under the $\mathcal{H}^2$ topology).
    \item For any function $u \in C^3(\overline\Omega)\cap C_0(\overline\Omega)$, the function $\tilde{u}=u/\eta$ is in $C^2_b(\Omega)\subset \mathcal{H}^2$.
\end{enumerate}
\end{lemma}

We now proceed to the proof of Lemma \ref{ker}.
\begin{proof}
By \eqref{40}, we have \begin{align}
    \langle h, \mathcal{S}h \rangle = \mathbb{E}_{c,w,b}\bigg[ \bigg(\int_{\Omega}h(x)\mathcal{A}[\eta(x)\sigma(x;w,b)]d\mu(x)\bigg)^2 \bigg]=0.
\end{align}
Therefore \begin{align} \label{52}
    \int_\Omega h(x)\mathcal{A}[\eta(x)\sigma(x;w,b)]d\mu(x)=0
\end{align}
for any $(c,w,b)$ by the continuity of the objective with respect to parameters $c,w,b$.

Since $\mu$ is a finite measure, from Theorem 4 in \cite{hornik1991approximation} we have that the linear span of $\{\sigma(w\cdot x+b)\}_{w,b\in \mathbb{R}}$ is dense in $\mathcal{H}^2$. For a general $f\in C^3(\bar{\Omega})\cap C_0(\bar \Omega)$, by Lemma \ref{lemma_eta}(ii), we can approximate the function $f/\eta$ within the linear span of $\{\sigma(w\cdot x+b)\}_{w,b\in \mathbb{R}}$. Multiplying by $\eta$, by Lemma \ref{lemma_eta}(i) it follows that the function class $\{\eta(x)\sigma(x;w,b)\}$ is dense in the function space $\mathcal{H}^2_{(0)}(\Omega)$, which consists of $\mathcal{H}^2$ function with boundary value zero. For any $f \in \mathcal{H}^2_{(0)}$, there exists a function sequence $\{F_N:\eta\sum_{i=1}^N c_i \sigma_{w,b} \}_{N \geq 1}$ such that \begin{align}
    \lim_{N \to \infty}\|F_N-f\|_{\mathcal{H}^2}=0.
\end{align}
Therefore, \begin{align}
    \langle h, \mathcal{A}f\rangle= \int_\Omega h(x)\mathcal{A}f d\mu(x)=\lim_{N \to \infty} \langle h, \mathcal{A}F^N \rangle =0.
\end{align}
\end{proof}
\subsection{Proof of Corollary \ref{dis}}
\begin{proof}
Suppose that $Q_*$ is a stationary point. Then we have \begin{align}\label{stationarybound}
    \mathcal{U}[\mathcal{A}Q^*-g]=0.
\end{align} Then \begin{align}
    \mathcal{S}\mathcal{A}[Q^*-u]=\mathcal{S}[\mathcal{A}Q^*-g]=\mathcal{A}\,\mathcal{U}[\mathcal{A}Q^*-g]=0.
\end{align}
By Lemma \ref{ker}, the inner product term $\langle \mathcal{A}[Q^*-u], \mathcal{A}f \rangle=0$ for any $f \in \mathcal{H}^2_{(0)}$. Therefore, by taking $f=Q^*-u$, we have $\|\mathcal{A}Q^*-g\|_2^2=\langle \mathcal{A}[Q^*-u], \mathcal{A}[Q^*-u] \rangle=0$.
At the same time, since $Q^* \in \mathcal{H}^2_{(0)}$ satisfies the zero boundary condition, we conclude that $Q_*=u$ is the solution of the PDE.
\end{proof}
\subsection{Proof of Theorem \ref{l51}}
\begin{proof}
   In \eqref{59}, multiplying $\mathcal{A}Q_t-g$ on both sides and integrating with respect to $\mu(dy)$:
   \begin{align}
    \frac{d\|\mathcal{A}[Q_t-u]\|_2^2}{dt}=-2\langle \mathcal{A}[Q_t-u], \mathcal{S}\mathcal{A}[Q_t-u] \rangle \leq 0
\end{align}
with strict inequality unless $\|\mathcal{A}Q_t-g\|_2=0$ which corresponds to PDE's solution. 

Consider $\mathcal{A}Q_t-g=\mathcal{A}[Q_t-u]$ projected on $\{\varepsilon_i\}_{i \in \mathbb{N}^+}\cup \{\nu_i\}_{i \in \mathbb{N}^+}$. By Lemma \ref{ker}, we have \begin{align}
    \mathcal{A}Q_t-g=\sum_{i} h_t^i \varepsilon_i + \sum_j 0 \nu_j=\sum_{i} h_t^i \varepsilon_i.
\end{align}
Therefore $\|AQ_t-g\|_2^2=\sum_i {h_t^i}^2$.
Now consider its projection on each $\varepsilon_i$ \begin{align}
    \frac{d}{dt}\langle \mathcal{A}Q_t-g, \varepsilon_i \rangle = \Big\langle \frac{d \mathcal{A}Q_t-g}{dt}, \varepsilon_i \Big\rangle = \langle -\mathcal{S} [\mathcal{A}Q_t-g], \varepsilon_i \rangle = \langle \mathcal{A}Q_t-g, -\mathcal{S}\varepsilon_i \rangle = -\lambda_i \langle \mathcal{A}Q_t-g, \varepsilon_i \rangle.
\end{align}
Therefore \begin{align}
    \frac{d}{dt}\langle \mathcal{A}Q_t-g, \varepsilon_i \rangle = \frac{d}{dt}[h_t^i]=-\lambda_i h_t^i.
\end{align}
Consequently $h_t^i = h_0^i e^{-\lambda_i t}$ whose absolute value decays exponentially, and $|h_t^i|\leq |h_0^i|$ for any $t\geq 0$ and any $i$. 

Now, by the dominated convergence theorem \begin{align}
    \lim_{t \to \infty} \|\mathcal{A}Q_t-g\|_2^2=\lim_{t\to \infty}\sum_{i}|h_t^i|^2=\sum_{i}\lim_{t\to \infty}|h_t^i|^2=0.
\end{align}
\end{proof}

\subsection{Proof of Theorem \ref{l62}}
\begin{proof}
        Writing $\mathcal{A}^{-1}$ for the inverse operator of $\mathcal{A}$, we have 
        \begin{align}
            \|Q_t-u\|_2 = \|\mathcal{A}^{-1}[\mathcal{A}[Q_t-u]]\|_2 \leq k \|\mathcal{A}[Q_t-u]\|_2=k \|\mathcal{A}Q_t-g\|_2 \to 0.
        \end{align}
    \end{proof}
\subsection{Proof of Lemma \ref{kerV}}
    \begin{proof}
    Denote by $\varrho = \E$ an eigenfunction of $\mathcal{V}$ which has an eigenvalue zero: \begin{align}
    \mathcal{V}\E=\mathcal{V}\varrho=0\varrho=\D\E.
\end{align}
We now show that \begin{align} \label{55}
    \bigg\langle \varrho, \B \bigg\rangle =0,
\end{align}
for any $f \in \mathcal{H}^2_{(0)}$. This is because, similar to \eqref{40}, we have \begin{align}
\begin{split}
    \big\langle \varrho, \mathcal{V}\varrho \big \rangle &\geq \mathbb{E}_{c,w,b}\bigg[ \bigg(\int_{\Omega}\varrho^d(x)\mathcal{A}[\eta(x)\sigma(x;w,b)]d\mu(x)+\int_{\Omega}\varrho^p(x)\eta(x)\sigma(x;w,b)d\mu_{\mathbf{x}}(x)\bigg)^2  \bigg] \geq 0.
\end{split} 
\end{align}
Therefore, $\mathcal{V}\varrho=0$ implies that for any $(w,b)$ pairs \begin{align}
        \int_{\Omega}\varrho^d(x)\mathcal{A}[\eta(x)\sigma(x;w,b)]d\mu(x)+
\int_{\Omega}\varepsilon^p(x)\eta(x)\sigma(x;w,b)d\mu_{\mathbf{x}}(x)=0.
\end{align}
Now, since for any $f-u \in \mathcal{H}^2_{(0)}$, there exists a function sequence $\{F_N:\eta\sum_{i=1}^N c_i \sigma_{w,b} \}_{N \geq 1}$ such that \begin{align}
    \lim_{N \to \infty}\|F_N-(f-u)\|_{\mathcal{H}^2}^2=0,
\end{align}
and simultaneously \begin{align} \label{59.2}
    \lim_{N \to \infty} \|F_N-(f-u)\|_{L^2(\mu(\mathbf{x}))}^2=\frac{1}{M}\sum_{i=1}^M [F_N(x_i)-(f(x_i)-u(x_i))]^2=0,
\end{align} 
we have \begin{align}
\int_{\Omega}\varrho^d(x)\mathcal{A}F^N(x)d\mu(x)+
\int_{\Omega}\varrho^p(x)F^N(x)d\mu_{\mathbf{x}}(x)=0.
\end{align}
This implies that \begin{align}
        \int_{\Omega}\varrho^d(x)\mathcal{A}[f-u](x)d\mu(x)+
\int_{\Omega}\varrho^p(x)[f-u](x)d\mu_{\mathbf{x}}(x)=0,
\end{align}
and hence \eqref{55} is proven.
\end{proof}
\subsection{Proof of Theorem \ref{l72}}
\begin{proof}
    From \eqref{pinns} we see that \begin{align}
    \frac{d}{dt}\bigg[ \|\mathcal{A}Q_t-g\|^2_{L^2(\mu)}+\|Q_t-u\|_{L^2(\mu_\mathbf{x})}^2 \bigg]=-\C^\intercal \mathcal{V} \C \leq 0,
\end{align}
and that the equality holds iff $Q_t=u$. Therefore, the optimization objective is decreasing.

Consider $\mathbf{\Tilde{Q_t}}=[\mathcal{A}Q_t-g, \mathbf{Q_t-u}]^\intercal$ projected on $\{\vartheta\}_{i \in \mathbb{N}^+}\cup \{\varrho_i\}_{i \in \mathbb{N}^+}$. By Lemma \ref{kerV}, we have \begin{align}
    \mathbf{\Tilde{Q_t}}=\sum_{i} h_t^i \vartheta_i + \sum_j 0 \varrho_j=\sum_{i} h_t^i \vartheta_i.
\end{align}
Therefore $\|\mathcal{A}Q_t-g\|^2_{L^2(\mu)}+\|Q_t-u\|_{L^2(\mu_\mathbf{x})}^2=\sum_i {h_t^i}^2$.
Now consider its projection on each $\vartheta_i$ \begin{align}
    \frac{d}{dt}\langle \mathbf{\Tilde{Q_t}}, \vartheta_i \rangle = \Big\langle \frac{d \mathbf{\Tilde{Q_t}}}{dt}, \vartheta_i \Big\rangle = \langle -\mathcal{V} \mathbf{\Tilde{Q_t}}, \vartheta_i \rangle = \langle \mathbf{\Tilde{Q_t}}, -\mathcal{V}\vartheta_i \rangle = -\lambda_i \langle \mathbf{\Tilde{Q_t}}, \vartheta_i \rangle.
\end{align}
Therefore \begin{align}
    \frac{d}{dt}\langle \mathbf{\Tilde{Q_t}}, \vartheta_i \rangle = \frac{d}{dt}[h_t^i]=-\lambda_i h_t^i.
\end{align}
Consequently $h_t^i = h_0^i e^{-\lambda_i t}$ whose absolute value decays exponentially, and $|h_t^i|\leq |h_0^i|$ for any $t\geq 0$ and any $i$. 

Now, by the dominated convergence theorem \begin{align}
    \lim_{t \to \infty} \|\mathcal{A}Q_t-g\|^2_{L^2(\mu)}+\|Q_t-u\|_{L^2(\mu_\mathbf{x})}^2=\lim_{t\to \infty}\sum_{i}|h_t^i|^2=\sum_{i}\lim_{t\to \infty}|h_t^i|^2=0.
\end{align}
\end{proof}

\subsection{Proof of Theorem \ref{l73}}
\begin{proof}
As the existence of an inverse guarantees that the PDE \eqref{pde1} admits a (unique) solution, it is clear from Theorem \ref{l72} that $\|Q_t-u\|_{L^2(\mu_\mathbf{x})} \to 0$ and $\|\mathcal{A}Q_t - g\|^2_{L^2(\mu)}\to 0$. As in the proof of Theorem \ref{l62}, the convergence of the residual implies the convergence of $Q_t$ to $u$, given  $\mathcal{A}^{-1}$ is a bounded operator.
\end{proof}
\bibliographystyle{plainnat}
\bibliography{mybib}

\end{document}